\documentclass[12pt,reqno,draft]{amsart}
\usepackage{indentfirst}
\usepackage{amsmath}
\usepackage{amsfonts}
\usepackage{amssymb}
\usepackage{pb-diagram}
\usepackage{graphicx}
\usepackage[all]{xy}%
\usepackage[a4paper,left=2cm,right=2cm,top=3cm,bottom=3cm]{geometry}

\newtheorem{theorem}{Theorem}
\theoremstyle{plain}

\newtheorem{lemma}{Lemma}

\newtheorem{proposition}{Proposition}

\numberwithin{equation}{section}

\begin{document}
\title[Lie Groupoids and Generalized Almost Subtangent Manifolds]{Lie Groupoids and Generalized Almost Subtangent Manifolds}
\author{Fulya Sahin}

\begin{abstract}
In this paper, we show that there is a close relationship between
generalized subtangent manifolds and Lie groupoids. We obtain
equivalent assertions among the integrability conditions of
generalized almost subtangent manifolds, the condition of
compatibility of source and target maps of symplectic groupoids with
symplectic form and generalized subtangent maps.
\end{abstract}

\maketitle

\section{Introduction}

 A groupoid is a small category in which all morphisms are invertible. More precisely, a groupoid $G$ consists of two sets $G_{1}$ and $G_{0}$, called arrows
 and the objects, respectively, with maps
 $s,t:G_{1}\rightarrow G_{0}$ called source and target. It is
 equipped with a composition $m:G_{2}\rightarrow G_{1}$
 defined on the subset $G_{2}=\{(g,h)\in G_{1}\times
 G_{1}|s(g)=t(h)\}$; an inclusion map of objects
 $e:G_{0}\rightarrow G_{1}$ and an inversion map $i:G_{1}\rightarrow G_{1}$.
 For a groupoid, the following properties are satisfied:
 $s(gh)=s(h)$, $t(gh)=t(g)$, $s(g^{-1})=t(g)$,
 $t(g^{-1})=s(g)$, $g(hf)=(gh)f$ whenever both sides are
 defined, $g^{-1}g=1_{s(g)}$, $gg^{-1}=1_{t(g)}$. Here we have
 used $gh, 1_{x}$ and $g^{-1}$ instead of $m(g,h)$, $e(x)$ and
 $i(g)$. Generally, a groupoid $G$ is denoted by the set of arrows
 $G_{1}$.

 A topological groupoid is a groupoid $G_{1}$ whose set of arrows
 and set of objects are both topological spaces whose structure maps
 $s, t, e, i,m$ are all continuous and $s, t$ are
 open maps.

 A Lie groupoid is a groupoid $G$ whose set of arrows and set of
 objects are both manifolds whose structure maps $s, t,
 e, i, m$ are all smooth maps and $s, t$ are submersions.
 The latter condition ensures that $s$ and $t$-fibres are
 manifolds. One can see from above definition the space $G_{2}$ of
 composable arrows is a submanifold of $G_{1}\times G_{1}$. We note
 that Lie groupoid introduced by Ehresmann \cite{ehr}.

 On the other hand, Lie algebroids were first introduced by
 Pradines \cite{prad} as infinitesimal objects associated with the Lie
 grou- poids. More precisely, a Lie algebroid structure on a real
 vector  bundle $A$ on a manifold $M$ is defined by a vector bundle
 map $\rho_{A}:A\rightarrow TM$, the anchor of $A$, and an
 $\mathbb{R}$-Lie algebra bracket on $\Gamma(A), [,]_{A}$ satisfying
 the Leibnitz rule
 \begin{center}
 $[\alpha,f \beta]_{A}=f[\alpha,\beta]_{A}+L_{\rho_{A}(\alpha)}(f)\beta$
 \end{center}
 for all $\alpha,\beta \in \Gamma(A), f \in C^{\infty}(M)$, where $L_{\rho_{A}(\alpha)}$ is
 the Lie derivative with respect to the vector field
 $\rho_{A}(\alpha)$. And $\Gamma(A)$ denotes the set of sections in
 $A$.

 On the other hand, Hitchin \cite{hitc} introduced the notion of generalized
 complex manifolds by unifying and extending the usual notions of
 complex and symplectic manifolds. Later the notion of generalized
 K\"{a}hler manifold was introduced by Gualtieri \cite{gualt} and submanifolds
 of such manifolds have been studied in many papers.

 As an analogue of generalized complex structures on even dimensional
 manifolds, the concept of  generalized almost subtangent manifolds
 were introduced in \cite{vaisman} and such manifolds have been
 studied in \cite{wade1} and \cite{vaisman}.

 Recently, Crainic \cite{crainic1} showed that there is a close relationship between the
 equations of a generalized complex manifold and a Lie groupoid. More precisely, he obtained that the complicated
 equations of such manifolds turn into simple structures for
 Lie groupoids.\\

 In this paper, we investigate relationships between the complicated
 equations of generalized subtangent structures and Lie groupoids.
 We showed that the equations of such manifolds are useful to obtain
 equivalent results on a symplectic groupoid.

\section{Preliminaries}

In this section we recall basic facts of Poisson geometry, Lie
groupoids and Lie algebroids. More details can be found in
\cite{mac1} and \cite{vaisman1}. A central idea in generalized
geometry is that $TM\oplus T^{\ast}M$ should be thought of as a
generalized tangent bundle to manifold $M$. If $X$ and $\xi$ denote
a vector field and  a dual vector field on $M$ respectively, then we
write $(X,\xi)$ (or $X+\xi$) as a typical element of $TM\oplus
T^{\ast}M$. The Courant bracket of two sections $(X,\xi),(Y,\eta)$
of $TM\oplus T^{\ast}M=\mathcal{TM}$ is defined by
\begin{eqnarray}
[(X,\xi),(Y,\eta)]&=&[X,Y]+L_{X}\eta-L_{Y}\xi\nonumber\\
& &\ -\frac{1}{2}d(i_{X}\eta-i_{Y}\xi),\label{eq:2.1}
\end{eqnarray}
where $d$, $L_{X}$ and $i_{X}$ denote exterior derivative, Lie
derivative and interior derivative with respect to $X$,
respectively. The Courant bracket is antisymmetric but, it does not
satisfy the Jacobi identity. We adapt the notions
$\beta(\pi^{\sharp}\alpha)=\pi(\alpha,\beta)$ and
$\omega_{\sharp}(X)(Y)=\omega(X,Y)$ which are defined as
$\pi^{\sharp}:T^{\ast}M \rightarrow TM$,
$\omega_{\sharp}:TM\rightarrow T^{\ast}M$  for any 1-forms $\alpha$
and $\beta$, 2-form $\omega$ and bivector field $\pi$, and vector
fields $X$ and $Y$. Also we denote by $[,]_{\pi}$, the bracket on
the space of 1-forms on $M$ defined by
\begin{eqnarray}
[\alpha,\beta]_{\pi}=L_{\pi^{\sharp}\alpha}\beta-L_{\pi^{\sharp}\beta}\alpha-d\pi(\alpha,\beta).\nonumber
\end{eqnarray}

On the other hand, a symplectic manifold is a smooth (even
dimensional) manifold $M$ with a non-degenerate closed 2-form
$\omega\in\Omega^{2}(M)$. $\omega$ is called the symplectic form of
$M$. Let $G$ be a Lie groupoid on $M$ and $\omega$ a form on Lie
groupoid $G$, then $\omega$ is called multiplicative if
\begin{eqnarray}
m^{\ast}\omega=pr_{1}^{\ast}\omega+pr_{2}^{\ast}\omega, \nonumber
\end{eqnarray}
where $pr_{i}:G\times G\rightarrow G$, $i=1,2$, are the canonical
projections.\\

We now recall the notion of Poisson manifolds. A Poisson manifold is
a smooth manifold $M$ whose function  space
$C^{\infty}(M,\mathbb{R})$, is a Lie algebra with bracket $\{,\}$,
such that the following properties are satisfied;
\begin{enumerate}
\item[(i)]$ \{f,g\}= -\{g,f\}$
\item[(ii)]$\{f,\{g,h\}\}+\{g,\{h,f\}\}+\{h,\{f,g\}\}=0$
\item[(iii)]$\{fg,h\}=f\{g,h\}+g\{f,h\}$.
\end{enumerate}
If $M$ is a Poisson manifold, then there is a unique bivector $\pi$,
called the Poisson bivector, and a unique homomorfizm
$\pi^{\sharp}:T^{\ast}M\rightarrow TM$ of vector bundles with
$\pi^{\sharp}(T^{\ast}M)\subset TM$ such that
$\pi(df,dg)=\pi^{\sharp}(df)g=\{f,g\}$.

It is also possible to define a Poisson manifold by using the
bivector $\pi$. Indeed, a smooth manifold is a Poisson manifold if
$[\pi,\pi]=0$, where $[,]$ denotes the Schouten bracket on the space
of multivector fields.

We now give a relation between Lie algebroid and Lie groupoi- d.
Given a Lie groupoid $G$ on $M$, the associated Lie algebroid
$A=Lie(G)$ has fibres $A_{x}=Ker(ds)_{x}=T_{x}(G(-,x))$, for any
$x\in M$. Any $\alpha \in \Gamma(A)$ extends to a unique
right-invariant vector field on $G$, which will be denoted by same
letter $\alpha$. The usual Lie bracket on vector fields induces the
bracket on $\Gamma(A)$, and the anchor is given by
$\rho=dt:A\rightarrow TM$.

Given a Lie algebroid $A$, an integration of $A$ is a Lie groupoid
$G$ together with an isomorphism $A\cong Lie(G)$. If  such a $G$
exists, then it is called that $A$ is integrable. In contrast with
the  case of Lie algebras, not every Lie algebroid admits an
integration. However if a Lie algebroid is integrable, then there
exists a canonical source simply connected integration $G$, and any
other source simply connected integration is smoothly isomorphic to
$G$. \textit{From now on we assume that all Lie groupoids are
source-simply-connected.}\\

In this section, finally, we recall the notion of $IM$ form
(infinitesimal multiplicative form) on a Lie algebroid
\cite{bursztyn}. More precisely, an $IM$ form on a Lie algebroid $A$
is a bundle map
\begin{eqnarray*}
u:A\rightarrow T^{\ast}M
\end{eqnarray*}
satisfying the following properties

\begin{enumerate}
\item[(i)]$\langle u(\alpha),\rho(\beta)\rangle=-\langle
u(\beta),\rho(\alpha)\rangle$
\item[(ii)] $u([\alpha,\beta])=L_{\alpha}(u(\beta))-L_{\beta}(u(\alpha))$\\
  $+d\langle u(\alpha),\rho(\beta)\rangle$
\end{enumerate}

for $\alpha,\beta \in \Gamma(A)$, where $\langle,\rangle$ denotes
the usual pairing between a vector space and its dual.

If $A$ is a Lie algebroid of a Lie groupoid $G$, then a closed
multiplicative 2-form $\omega$ on $G$ induces an $IM$ form
$u_{\omega}$ of $A$ by
\begin{eqnarray}
\langle u_{\omega}(\alpha),X \rangle=\omega(\alpha,X). \nonumber
\end{eqnarray}
For the relationship between $IM$ form and closed 2-form we have the
following.
\begin{theorem}
\cite{bursztyn}If $A$ is an integrable Lie algebroid and if $G$ is
its integration, then $\omega\mapsto u_{\omega}$ is an one to one
correspondence between closed multiplicative 2-forms on $G$ and IM
forms of $A$.
\end{theorem}

If a Lie groupoid $G$ is endowed with a symplectic form which is
multiplicative, then $G$ is called  symplectic groupoid.

Similar to 2-forms,  given a Lie groupoid $G$, a (1,1)-tensor $J :
TG \rightarrow TG$ is called multiplicative \cite{crainic1} if for
any $(g, h)\in G\times G$ and any $v_{g}\in T_{g}G$, $w_{h}\in
T_{h}G$ such that $(v_{g},w_{h})$ is tangent to $G\times G$ at $(g,
h)$, so is $(Jv_{g}, Jw_{h})$, and
\begin{eqnarray*}
(dm)_{g,h}(Jv_{g}, Jw_{h}) = J((dm)_{g,h}(v_{g},w_{h} )).
\end{eqnarray*}
\section{Lie Groupoids and Generalized Subtangent Structures}

In this section we first give a characterization for generalized
subtangent manifolds, then we obtain certain relationships between
generalized subtangent manifolds and symplectic groupoi- ds. We
recall that a generalized almost subtangent structure $\mathcal{J}$
is an endomorphism on $\mathcal{TM}$ such that $\mathcal{J}^{2}= 0$.
A generalized almost subtangent structure can be represented by
classical tensor fields as follows:
\begin{eqnarray}
\mathcal{J}= \left[ {\begin{array}{cc}
 a & \pi^{\sharp}  \\ \label{eq:3.1}
 \sigma_{\sharp} & -a^\ast  \\
 \end{array} } \right]
\end{eqnarray}
where $\pi$ is a bivector on $M$, $\sigma$ is a 2-form on $M$, $a :
TM \rightarrow TM$ is a bundle map, and $a^{\ast} : T^{\ast}M
\rightarrow T^{\ast}M$ is dual of $a$, for almost subtangent
structures see:\cite{vaisman} and \cite{wade1}.

A  generalized almost  subtangent structure is called integrable (or
just subtangent structure) if $\mathcal{J}$ satisfies the following
condition
\begin{eqnarray}
[\mathcal{J}\alpha,\mathcal{J}\beta]-
\mathcal{J}([\mathcal{J}\alpha, \beta] + [\alpha,\mathcal{J}\beta])
= 0, \label{eq:3.2}
\end{eqnarray}
 for all sections $\alpha,\beta \in \mathcal{TM}$.\\

In the sequel, we give necessary and sufficient conditions for a
generalized almost subtangent structure to be integrable in terms of
the above tensor fields. We note that the following result was
stated in \cite{vaisman}, but its proof was not given in there. In
fact, the proof of the conditions of the following proposition is
similar to the proposition given in  \cite{crainic1} by Crainic for
generalized complex structures. Although the conditions are similar
to the generalized complex case, their proofs are slightly different
from the complex case. Therefore we give one part of the proof of
the following proposition.
\begin{proposition}A manifold with $\mathcal{J}$ given by
(\ref{eq:3.1}) is a generalized subtangent manifold if and only if
\begin{enumerate} \item [(S1)] $\pi$ satisfies the equation
\begin{equation}
\pi^{\sharp}([\xi,\eta]_{\pi})=[\pi^{\sharp}(\xi),\pi^{\sharp}(\eta)],
\nonumber
\end{equation}
\item [(S2)]$\pi$ and $a$ are related by the following two
formulas
\begin{eqnarray}
a\pi^{\sharp}&=&\pi^{\sharp}a^{\ast},\nonumber\\
a^{\ast}([\xi,\eta]_{\pi}) &=& L_{\pi^{\sharp}\xi}(a^{\ast}\eta) -
L_{\pi^{\sharp}\eta}(a^{\ast}\xi)\nonumber\\
& &\ -d\pi(a^{\ast}\xi,\eta),\label{eq:3.5}
\end{eqnarray}
\item [(S3)]$\pi,a$ and $\sigma$ are related by the following two
formulas
\begin{eqnarray}
a^{2} + \pi^{\sharp}\sigma_{\sharp} &=& 0,\label{eq:3.6}\\
N_{a}(X,Y)&=&\pi^{\sharp}(i_{X\wedge Y}d(\sigma)),\label{eq:3.7}
\end{eqnarray}
\item [(S4)]$\sigma$ and $a$ are related by the following two
formulas
\begin{eqnarray}
a^{\ast}\sigma_{\sharp} &=& \sigma_{\sharp}a, \nonumber\\
d\sigma_{a}(X,Y,Z) &=& d\sigma(aX,Y,Z) + d\sigma(X,aY,Z)\nonumber\\&
&\ + d\sigma(X, Y,aZ)\nonumber
\end{eqnarray}
for all 1-forms $\xi$ and $\eta$, and all vector fields $X, Y$ and
$Z$, where $\sigma_{a}(X,Y) = \sigma(aX,Y)$.
\end{enumerate}
\end{proposition}

As an analogue of a Hitchin pair on generalized complex manifold, a
Hitchin pair on a  generalized almost subtangent manifold $M$ is a
pair $(\omega, a)$ consisting of a symplectic form $\omega$ and a
(1,1)-tensor $a$ with the property that $\omega$ and $a$ commute
(i.e $\omega(X,aY)=\omega(aX,Y)$) and $d\omega_{a} = 0$, where
$\omega_{a}(X,Y)=\omega(aX,Y)$.
\begin{lemma}
If $\pi$ is a non-degenerate bivector on a generalized almost
subtangent manifold $M$, $\omega$ is the inverse 2-form (defined by
$\omega_{\sharp} = (\pi^{\sharp})^{-1}$) and $\pi$ satisfies
(\ref{eq:3.6}) then $\sigma=-a^{\ast}\omega$.
\end{lemma}
\begin{proof}
For $X\in\chi(M)$, we apply $\omega_{\sharp}$ to (\ref{eq:3.6}) and
using the dual subtangent structure $a^{\ast}$, we have
\begin{eqnarray}
(a^{\ast})^{2}\omega_{\sharp}(X)+\sigma_{\sharp}(X)=0.\nonumber
\end{eqnarray}
Now for $Y\in\chi(M)$, since $\omega$ and $a$ are commute, we obtain
\begin{equation}
\omega(aX,aY)+\sigma(X,Y)=0.\nonumber
\end{equation}
Thus we get
\begin{equation}
a^{\ast}\omega(X,Y)+\sigma(X,Y)=0.\label{eq:3.18}
\end{equation}
Since the equation (\ref{eq:3.18}) is hold for all $X$ and $Y$, we
get
\begin{eqnarray}
\sigma=-a^{\ast}\omega.\nonumber
\end{eqnarray}
\end{proof}
We say that 2-form $\sigma$ is the twist of Hitchin pair
$(\omega,a)$.\\

A symplectic+subtangent structure on a generalized almost subtangent
manifold $M$ consists of a pair $(\omega, J)$ with
$\omega$-symplecti- c and $J$-subtangent structure on $M$, which
commute.
\begin{lemma}
Let $(M,\omega)$ be a symplectic manifold. $(\omega, a)$ is a
symplectic+subtangent structure if and only if $d\omega_{a} = 0$, $
a^{\ast}\omega = 0$.
\end{lemma}
\begin{proof}
We will only prove the sufficient condition. Since $(M,\omega)$ is a
symplectic manifold, then $d\omega=0$. Since $d\omega_{a} = 0$,
 $a^{\ast}\omega=0 $, by using the following equation (see
\cite{crainic1}),
\begin{eqnarray}
i_{N_{a}(X,Y)}(\omega)&=&i_{aX\wedge Y+X\wedge
aY}(d\omega_{a})-i_{aX\wedge aY}(d\omega)\nonumber\\
& &\ -i_{X\wedge
Y}(d(a^{\ast}\omega)),\label{eq:3.19}
\end{eqnarray}
we get $i_{N_{a}(X,Y)}(\omega)=-i_{X\wedge Y}(d(a^{\ast}\omega))=0$.
Hence, $$\omega(N_{a}(X,Y),\bullet)=0.$$ Since $\omega$ is
non-degenerate, then $N_{a}=0$. Thus $a$ is a subtangent structure.
On the other hand, $a^{\ast}\omega = 0$ implies that $\omega$ and
$a$ commute.

The converse is clear.
\end{proof}
Next we relate (S1) and the 2-form $\omega$.
\begin{lemma}
If $\pi$ is a non-degenerate bivector on a generalized almost
subtangent manifold $M$, and $\omega$ is the inverse 2-form (defined
by $\omega_{\sharp} = (\pi^{\sharp})^{-1}$), then $\pi$ satisfies
(S1) if and only if $\omega$ is closed.
\end{lemma}
\begin{proof}
Applying  $\xi=i_{X}(\omega)$ and $\eta=i_{Y}(\omega)$ to (S1), we
get
\begin{eqnarray}
\pi^{\sharp}(L_{X}(i_{Y}(\omega))-L_{Y}(i_{X}(\omega))-d(\omega_{\sharp}Y(\pi^{\sharp}\omega_{\sharp}X))=[X,Y]\nonumber
\end{eqnarray}
Since $\omega_{\sharp}=(\pi^{\sharp})^{-1}$, we have
\begin{eqnarray}
\pi^{\sharp}(L_{X}(i_{Y}(\omega))-L_{Y}(i_{X}(\omega))-d(\omega_{\sharp}Y)(X))=[X,Y]\label{eq:3.20}
\end{eqnarray}
Then applying $\omega_{\sharp}$ to (\ref{eq:3.20}), we derive
\begin{eqnarray}
L_{X}(i_{Y}(\omega))-L_{Y}(i_{X}(\omega))-d(\omega_{\sharp}Y)(X)=\omega_{\sharp}([X,Y])\nonumber
\end{eqnarray}
Using
\begin{eqnarray} i_{X\wedge
Y}(d\sigma)&=&L_{X}(i_{Y}\sigma)-L_{Y}(i_{X}\sigma)+d(i_{X\wedge
Y}\sigma)\nonumber\\& &\ - i_{[X,Y]}\sigma, \label{eq:3.14}
\end{eqnarray} formula, then we get
\begin{eqnarray}
L_{X}(i_{Y}\omega)-L_{Y}(i_{X}\omega)+d(i_{X\wedge
Y}\omega)-i_{[X,Y]}\omega=i_{X\wedge Y}d\omega.\nonumber
\end{eqnarray}
Since left hand side vanishes in above equation, then $i_{X\wedge
Y}d\omega=0$. Thus $d\omega=0$. The converse is clear.
\end{proof}

Thus, we have the following result which shows that there is close
relationship between condition (S1) and a symplectic groupoid.
\begin{theorem}
Let $M$ be a generalized almost subtangent manifold. There is a 1-1
correspondence between:
\begin{enumerate}
\item[(i)]Integrable bivectors $\pi$ on $M$ satisfying (S1),
\item[(ii)]Symplectic groupoids $(\Sigma,\omega)$ over $M$.
\end{enumerate}
\end{theorem}
Since $\pi^{\sharp}$ and $[,]_{\pi}$ define a Lie algebroid
structure on $T^{\ast}M$, one can obtain the above theorem by
following the steps given in (\cite{crainic1},Theorem 3.2). We now
give the conditions for (S2) in terms of $\omega$ and $\omega_{a}$.
\begin{lemma}
Let $M$ be a generalized  almost subtangent manifold and $\omega$ a
symplectic form. Given a non-degenerate bivector $\pi$ (i.e.
$\pi^{\sharp}=(\omega_{\sharp})^{-1}$) and  a map $a:TM\rightarrow
TM$, then $\pi$ and $a$ satisfy (S2) if and only if $\omega$ and $a$
commute and $\omega_{a}$ is closed.
\end{lemma}
\begin{proof}
For a 1-form $\xi$, we  use $\xi=i_{X}\omega=\omega_{\sharp}X$ such
that $X$ is an arbitrary vector field. Since
\begin{eqnarray}
a\pi^{\sharp}(i_{X}\omega)&=&\pi^{\sharp}a^{\ast}(i_{X}\omega),\label{eq:3.21}
\end{eqnarray}
applying $\omega_{\sharp}$ to (\ref{eq:3.21}) and using
$\pi^{\sharp}=(\omega_{\sharp})^{-1}$; for a vector field $Y$, we
have
\begin{eqnarray}
\omega(aX,Y)=a^{\ast}(i_{X}\omega)(Y),\nonumber
\end{eqnarray}
which gives
\begin{eqnarray}
\omega(aX,Y)&=&\omega(X,aY).\nonumber
\end{eqnarray}
Let $\xi$ and $\eta$ be 1-forms such that
$\xi=i_{X}\omega=\omega_{\sharp}(X)$ and
$\eta=i_{Y}\omega=\omega_{\sharp}(Y)$ for  arbitrary vector fields
$X$ and $Y$. Then from (\ref{eq:3.5}) we have
\begin{eqnarray}
a^{\ast}(L_{\pi^{\sharp}\omega_{\sharp}(X)}(\omega_{\sharp}(Y))-L_{\pi^{\sharp}\omega_{\sharp}(Y)}(\omega_{\sharp}(X))-d\pi(i_{X}\omega,i_{Y}\omega))\nonumber\\
= L_{\pi^{\sharp}\omega_{\sharp}(X)}(a^{\ast}\omega_{\sharp}(Y))\nonumber\\
 -L_{\pi^{\sharp}\omega_{\sharp}(Y)}(a^{\ast}\omega_{\sharp}(X))\nonumber\\
 -d\pi(a^{\ast}i_{X}\omega,i_{Y}\omega).\nonumber
\end{eqnarray}
Since $\pi^{\sharp}=(\omega_{\sharp})^{-1}$ and
$a^{\ast}i_{Y}\omega=i_{Y}\omega_{a}$, we obtain
\begin{eqnarray}
a^{\ast}(L_{X}(i_{Y}\omega)-L_{Y}(i_{X}\omega)-d(i_{Y\wedge
X}\omega))&=& L_{X}(i_{Y}\omega_{a})\nonumber\\
&-&L_{Y}(i_{X}\omega_{a})\nonumber\\
&-&d(i_{Y}\omega(\pi^{\sharp}(\omega_{\sharp}(aX)))).\nonumber
\end{eqnarray}
Using (\ref{eq:3.14}) for left hand side, then we get
\begin{eqnarray}
a^{\ast}(i_{X\wedge Y}(d\omega)+i_{[X,Y]}\omega)=
L_{X}(i_{Y}\omega_{a}) -L_{Y}(i_{X}\omega_{a})
-d(\omega(Y,aX)).\nonumber
\end{eqnarray}
Since $\omega$ and $a$ commute, using also this for the right hand
side, we have
\begin{eqnarray}
a^{\ast}(i_{X\wedge Y}(d\omega))+i_{[X,Y]}\omega_{a}=i_{X\wedge
Y}(d\omega_{a})+i_{[X,Y]}\omega_{a}.\label{eq:3.23}
\end{eqnarray}
Since $\omega$ is closed, (\ref{eq:3.23}) implies that $ i_{X\wedge
Y}(d\omega_{a})=0$, i.e. $\omega_{a}$ is closed.
\end{proof}
Note that it is well known that there is one to one correspondence
between (1,1)-tensors $a$ commuting with $\omega$ and 2-forms on
$M$. On the other hand, it is easy to see that (S2) is equivalent to
the fact that $a^{\ast}$ is an $IM$ form on the Lie algebroid
$T^{\ast}M$ associated Poisson structure $\pi$. Thus from the above
discussion, Lemma 4 and Theorem 1, one can conclude the following
theorem.
\begin{theorem}\label{the:3}
Let $M$ be a generalized  almost subtangent manifold. Let $\pi$ be
an integrable Poisson structure on $M$, and  $(\Sigma,\omega)$ the
symplectic groupoid over $M$. Then there is a natural 1-1
correspondence between
\begin{enumerate}
\item[(i)] (1,1)-tensors $a$ on $M$ satisfying (S2).
\item[(ii)]multiplicative (1,1)-tensors $J$ on $\Sigma$ with the property that $(J,\omega)$ is a Hitchin
pair.
\end{enumerate}
\end{theorem}

We recall the notion of generalized subtangent map between
generalized subtangent manifolds. This notion was given in
\cite{vaisman} similar to the generalized subtangent map between
generalized complex manifolds given in \cite{crainic1}.

Let $(M_{i},\mathcal{J}_{i})$, $i={1,2}$, be two generalized
subtangent manifolds, and let $a_{i}, \pi_{i}, \sigma_{i}$ be the
components of $\mathcal{J}_{i}$ in the matrix representation
(\ref{eq:3.1}). A map $f : M_{1}\rightarrow M_{2}$ is called
generalized subtangent iff $f$ maps $\pi_{1}$ into $\pi_{2}$,
$f^{\ast}\sigma_{2} =\sigma_{1}$ and $(df)\circ a_{1} = a_{2}\circ
(df)$\cite{vaisman}.

We now state and prove the main result of this paper. This result
gives equivalent assertions between the condition (S3), twist
$\sigma$ of $(\omega,J)$ and subtangent maps for a symplectic
groupoid over $M$.
\begin{theorem}
Let $M$ be a generalized  almost subtangent manifold and
$(\Sigma,\omega,J)$ the induced symplectic groupoid over $M$ with
the induced multiplicative (1,1)-tensor. Assume that $(\pi,J)$
satisfy (S1), (S2) with integrable $\pi$. Then for a 2-form on $M$,
the following assertions are equivalent.
\begin{enumerate}
\item [(i)] (S3) is satisfied,
\item [(ii)] $-J^{\ast}\omega=t^{\ast}\sigma- s^{\ast}\sigma$,
\item [(iii)] $(t,s):\Sigma\rightarrow M\times \overline{M}$ is
generalized subtangent map; (condition of generalized subtangent map
on $M$ is $(dt)\circ a_{1} = a_{2}\circ (dt)$, this condition on
$\overline{M}$ is $(ds)\circ a_{1} = -a_{2}\circ (ds)$).
\end{enumerate}
\end{theorem}
\begin{proof}
(i)$\Leftrightarrow$ (ii). Define
$\phi=\widetilde{\sigma}-t^{\ast}\sigma+ s^{\ast}\sigma$, such that
$\widetilde{\sigma}=-J^{\ast}\omega$ and $A=ker(ds)|_{M}$. We know
from Theorem 1 that closed multiplicative 2-form $\theta$ on
$\Sigma$ vanishes if and only if $IM$ form $u_{\theta}=0$, i.e.
$\theta(X,\alpha)=0$, such that $X\in TM$, $\alpha\in A$. This case
can be applied for forms with high dimension, i.e. 3-form $\theta$
vanishes if and only if $\theta(X,Y,\alpha)=0$.

Since $\omega$ and $\omega_{J}$ are closed, from (\ref{eq:3.19}) we
get $i_{X\wedge Y}(d(J^{\ast}\omega))=-i_{N_{J}(X,Y)}\omega$.
Putting $\widetilde{\sigma}=-J^{\ast}\omega$, we obtain
\begin{equation}
i_{X\wedge
Y}(d\widetilde{\sigma})=i_{N_{J}(X,Y)}\omega.\label{eq:3.24}
\end{equation}
Since $d\phi=0 \Leftrightarrow d\phi(X,Y,\alpha)=0$, we have
\begin{eqnarray}
d\phi(X,Y,\alpha)=0 \Leftrightarrow
d\widetilde{\sigma}(X,Y,\alpha)-d(t^{\ast}\sigma)(X,Y,\alpha)\nonumber\\
+d(s^{\ast}\sigma)(X,Y,\alpha)=0\nonumber
\end{eqnarray}
On the other hand, we obtain
\begin{equation}
 d(t^{\ast}\sigma)(X,Y,\alpha)=d\sigma(dt(X),dt(Y),
 dt(\alpha)).\label{eq:3.25}
\end{equation}
If we take  $dt=\rho$ in (\ref{eq:3.25}) for $A$, we get
\begin{equation}
d(t^{\ast}\sigma)(X,Y,\alpha)=d\sigma(dt(X),dt(Y),
\rho(\alpha)).\label{eq:3.26}
\end{equation}
On the other hand, from \cite{bursztyn} we know that
\begin{equation}
Id_{\Sigma}=m\circ(t,Id_{\Sigma}).\label{eq:3.27}
\end{equation}
Differentiating (\ref{eq:3.27}), we obtain
\begin{eqnarray}
X=dt(X).\label{eq:3.28}
\end{eqnarray}
Using (\ref{eq:3.28}) in (\ref{eq:3.26}), we get
\begin{equation}
d(t^{\ast}\sigma)(X,Y,\alpha)=d\sigma(X,Y, \rho(\alpha)).\nonumber
\end{equation}
In a similar way, we see that
\begin{equation}
 d(s^{\ast}\sigma)(X,Y,\alpha)=d\sigma(ds(X),ds(Y),
 ds(\alpha)).\nonumber
\end{equation}
Since $\alpha\in kerds$, then $ds(\alpha)=0$. Hence
$d(s^{\ast}\sigma)=0$.\\
Thus
\begin{equation}
d\widetilde{\sigma}(X,Y,\alpha)=d\sigma(X,Y,\rho(\alpha)).\label{eq:3.29}
\end{equation}
Using (\ref{eq:3.24}) in (\ref{eq:3.29}), we derive
\begin{equation}
\omega(N_{J}(X,Y),\alpha)=d\sigma(X,Y,\rho(\alpha)).\label{eq:3.30}
\end{equation}
On the other hand, it is clear that $\phi=0\Leftrightarrow
\widetilde{\sigma}-t^{\ast}\sigma+ s^{\ast}\sigma=0$. Thus we obtain
\begin{eqnarray}
\widetilde{\sigma}(X,\alpha)=\sigma(X,\rho(\alpha)).\nonumber
\end{eqnarray}
Since $\widetilde{\sigma}=-J^{\ast}\omega$, we get
\begin{equation}
-\omega(JX,J\alpha)=\sigma(X,\rho(\alpha)).\label{eq:3.31}
\end{equation}
Since  Poisson bivector $\pi$ is integrable, it defines a Lie
algebroid whose anchor map is $\rho=\pi^{\sharp}$. Let us use
$\pi^{\sharp}$ instead of $\rho$ in (\ref{eq:3.30}) and
(\ref{eq:3.31}). Then we get
\begin{equation}
\omega(N_{J}(X,Y),\alpha)=d\sigma(X,Y,\pi^{\sharp}(\alpha)),\label{eq:3.32}
\end{equation}
\begin{equation}
-\omega(JX,J\alpha)=\sigma(X,\pi^{\sharp}(\alpha)).\nonumber
\end{equation}
Since $\omega(\alpha,X)=\alpha(X)$,
$\omega_{J}(\alpha,X)=\alpha(JX)$, from (\ref{eq:3.32}) we have
\begin{eqnarray}
-\alpha(N_{J}(X,Y))&=&d\sigma(X,Y,\pi^{\sharp}(\alpha))\nonumber\\
&=&i_{X\wedge Y}d\sigma(\pi^{\sharp}(\alpha))\nonumber\\
&=&\pi(\alpha,i_{X\wedge Y}d\sigma)\nonumber\\
&=&-\alpha(\pi^{\sharp}(i_{X\wedge Y}d\sigma)).\nonumber
\end{eqnarray}
i.e. $\alpha(N_{J}(X,Y))=\alpha(\pi^{\sharp}(i_{X\wedge
Y}d\sigma))$.\\
Since above equation is hold for all  non-degenerate $\alpha$, we
get
\begin{eqnarray}
N_{J}(X,Y)=\pi^{\sharp}(i_{X\wedge Y}d\sigma).\label{eq:3.33}
\end{eqnarray}
 On the other hand, from (\ref{eq:3.31}) we obtain

\begin{eqnarray}
\alpha(a^{2}X)&=&i_{X}\sigma(\pi^{\sharp}(\alpha))\nonumber\\
&=&\pi(\alpha,i_{X}\sigma)\nonumber\\
&=&-\alpha(\pi^{\sharp}\sigma_{\sharp}X).\nonumber
\end{eqnarray}
Thus we get
\begin{eqnarray}
a^{2}+\pi^{\sharp}\sigma_{\sharp}=0.\label{eq:3.34}
\end{eqnarray}
Then (i)$\Leftrightarrow$(ii) follows from (\ref{eq:3.33}) and
(\ref{eq:3.34}).\\

(ii)$\Leftrightarrow$(iii) $-J^{\ast}\omega=t^{\ast}\sigma-
s^{\ast}\sigma$ says $(t,s)$ is compatiple with 2-forms. Also it is
clear that $(t,s)$ and bivectors are compatible due to $\Sigma$ is a
symplectic groupoid. We will check the compatibility of $(t,s)$ and
(1,1)-tensors. From compatibility condition of $t$ and $s$, we will
get $dt\circ
J=a\circ dt$ and $ ds\circ J=-a\circ ds$.\\
For all $\alpha\in A$ and $V\in \chi(\Sigma),$
\begin{eqnarray}
\omega(\alpha,V)=\omega(\alpha,dt V)\nonumber
\end{eqnarray}
which is equivalent to
\begin{eqnarray}
\alpha(V)=\langle u_{\omega}(\alpha),dt V \rangle.\nonumber
\end{eqnarray}
Since $u_{\omega}=Id$ and $u_{\omega_{J}}=a^{\ast}$, we get
\begin{eqnarray}
\langle \alpha,a (dt (V)) \rangle&=&\alpha(a (dt (V)))\nonumber\\
&=&\omega(\alpha,dt (JV))\nonumber\\
&=&\langle \alpha,dt (JV)\rangle.\nonumber
\end{eqnarray}
Since this equation is hold for all $\alpha$, then  $a (dt)=dt (J)$.
Using $s=t\circ i$,
\begin{eqnarray}
a (ds (V))&=&a d(t\circ i) V\nonumber\\
&=&-ds (JV),\nonumber
\end{eqnarray}
which shows that $a (ds)=-ds(J)$. Thus proof is completed.
\end{proof}

\bigskip
\indent\textbf{Corresponding Author:}
\newline \indent Fulya \c{S}AH\.{I}N,

\indent Inonu University
\newline \indent Faculty of Science and Art\newline
\indent Department of Mathematics\newline \indent
 $44280$ Malatya-TURKEY\newline
\indent e-mail: fulya.sahin@inonu.edu.tr


\begin{thebibliography}{99}

\bibitem {ehr}Ehresmann, C.: Cat\`{e}gories topologiques et cat\`{e}gories differentiables, Coll. Geom. Diff.
Globales, Bruxelles, 137-150, (1959).

\bibitem {prad}Pradines, J.: 'Th\'{e}orie de Lie
pour les groupo\"{\i}des diff\'{e}rentiables, Calcul
diff\'{e}rentiel dans la cat\'{e}gorie des groupo\"{\i}des
infinit\'{e}simaux', Comptes rendus Acad. Sci. Paris 264 A, 245-248,
(1967).

\bibitem {hitc}Hitchin,N.: Generalized Calabi-Yau manifolds, Q. J.
Math., 54,  281–308, (2003).

\bibitem {gualt}Gualtieri, M.: Generalized complex geometry, Ph.D. thesis, Univ. Oxford, arXiv:math.DG/0401221, (2003).


\bibitem {vaisman}Vaisman, I.: Reduction and submanifolds of generalized complex manifolds, Dif. Geom. and its appl.,
25, 147-166, (2007).

\bibitem {wade1}Wade, A.: Dirac structures and paracomplex manifolds, C. R. Acad. Sci.
Paris, Ser. I, 338, 889-894, (2004).


\bibitem {crainic1}Crainic M.: Generalized complex structures and Lie brackets, Bull. Braz. Math. Soc., New Series 42(4), 559-578
(2011).

\bibitem {mac1}Mackenzie K.: Lie groupoids and Lie algebroids in differential geometry,
London Mathematical Society Lecture Note Series, vol. 124, Cambridge
University Press, Cambridge, (1987).

\bibitem {vaisman1}Vaisman, I.: Lectures on the geometry of Poisson manifolds, Progress in
Math., vol. 118, Birkh\"{a}user Verlag, Boston, (1994).

\bibitem {bursztyn}Bursztyn, H., Crainic, M., Weinstein, A., Zhu, C.: Integration of twisted Dirac brackets, Duke Math.
J., 123, 549-607, (2004).




\end{thebibliography}
\end{document}